\documentclass[letterpaper,12pt]{amsart}

\title{Completion of tree metrics and rank 2 matrices}

\author{Daniel Irving Bernstein}
\address{Department of Mathematics \\ North Carolina State University, Raleigh, NC 27695}
\email{dibernst@ncsu.edu}

\usepackage{latexsym,array,delarray,amsthm,amssymb,epsfig,graphicx,multirow,longtable}
\usepackage[numbers]{natbib}
\usepackage{subcaption}
\usepackage{tikz}
\usepackage{verbatim}
\usepackage{graphicx}
\usepackage{mathdots}

\theoremstyle{plain}
\newtheorem{thm}{Theorem}[section]
\newtheorem{lemma}[thm]{Lemma}
\newtheorem{prop}[thm]{Proposition}

\newtheorem*{thm*}{Theorem}
\newtheorem*{lemma*}{Lemma}
\newtheorem*{prop*}{Proposition}
\newtheorem*{cor*}{Corollary}
\newtheorem*{conj*}{Conjecture}

\theoremstyle{definition}
\newtheorem{defn}[thm]{Definition}
\newtheorem*{defn*}{Definition}

\theoremstyle{remark}
\newtheorem{rmk}[thm]{Remark}

\usepackage{thmtools, thm-restate}

\newcommand{\qq}{\mathbb{Q}}
\newcommand{\rr}{\mathbb{R}}
\newcommand{\cc}{\mathbb{C}}
\newcommand{\kk}{\mathbb{K}}

\newcommand{\calm}{\mathcal{M}}
\newcommand{\cals}{\mathcal{S}}

\newcommand{\calt}{\mathcal{T}}

\newcommand{\ind}{\mbox{$\perp \kern-5.5pt \perp$}}

\newcommand{\puis}{\cc\{\!\{t\}\!\}}
\newcommand{\val}{\textnormal{val}}
\newcommand{\trop}{\textnormal{trop}}
\newcommand{\rowspan}{\textnormal{rowspan}}
\newcommand{\cat}{\textnormal{Cat}}
\renewcommand{\span}{\textnormal{span}}
\newcommand{\gr}{\textnormal{Gr}}

\let\originalleft\left
\let\originalright\right
\def\left#1{\mathopen{}\originalleft#1}
\def\right#1{\originalright#1\mathclose{}}

\tikzstyle{vertex}=[circle, draw, inner sep=0pt, minimum size=6pt, fill=black]
\newcommand{\vertex}{\node[vertex]}

\begin{document}

\begin{abstract}
    Motivated by applications to low-rank matrix completion,
    we give a combinatorial characterization of the independent sets in
    the algebraic matroid associated to the collection of $m\times n$ rank-2 matrices
    and $n\times n$ skew-symmetric rank-2 matrices.
    Our approach is to use tropical geometry to reduce this to
    a problem about phylogenetic trees which we then solve.
    In particular, we give a combinatorial description of the collections
    of pairwise distances between several taxa that may be arbitrarily prescribed
    while still allowing the resulting dissimilarity map to be completed to a tree metric.


    \smallskip
    \noindent \textbf{MSC Classes:} 14T05, 52B40, 52C25
\end{abstract}

\maketitle

\section{introduction}\label{sec:intro}

Given a matrix where only some of the entries are known,
the low-rank matrix completion problem is to determine the missing
entries under the assumption that the matrix has some low rank $r$.
One can also assume additional structure such as (skew) symmetry or positive definiteness.
Practical applications of the low-rank matrix completion problem abound.
A well-known example is the so-called ``Netflix Problem'' of predicting an individual's
movie preferences from ratings given by several other users.
A brief survey of other applications appears in \cite{candes2010matrix}.
\\
\indent
Singer and Cucuringu show how ideas from rigidity theory can be applied to this
problem in \cite{singer2010uniqueness}.
Jackson, Jord{\'a}n, and Tibor further develop these ideas
in \cite{jackson2014combinatorial,jackson2016unique}.
Kir{\'a}ly, Theran, and Tomioka incorporate ideas from algebraic geometry 
into this rigidity-theoretic framework in \cite{kiraly-theran-tomioka2015} and
Kir{\'a}ly, Theran and Rosen further develop these ideas in \cite{kiraly-rosen-theran2013}.
We add tools from tropical geometry to this picture.
\\
\indent
Let $V$ be a determinantal variety over some algebraically closed field $\kk$.
The results in this paper concern the cases were $V = \cals_r^n(\kk)$,
the collection of $n\times n$ skew-symmetric $\kk$-matrices of rank at most $r$,
or $V = \calm_{r}^{m\times n}(\kk)$,
the collection of $m\times n$ $\kk$-matrices of rank at most $r$.
A \emph{masking operator} corresponding to some $S \subseteq \binom{[n]}{2}$
in the skew symmetric case,
or $S \subseteq [m]\times [n]$ in the rectangular case,
is a map $\Omega_S : V \rightarrow \kk^S$
that projects a matrix $M$ onto the entries specified by $S$.
In the case of skew-symmetric $n\times n$ matrices,
we view $S$ as the edge set of a graph on vertex set $[n]$,
which we denote $G(S)$.
In the case of rectangular matrices,
we view $S$ as the edge set of a bipartite graph on partite sets of size $m$ and $n$
which we also denote $G(S)$.
Context will make the proper interpretation of $G(S)$ clear.
\\
\indent
Low-rank matrix completion problems can now be phrased as:
given $\Omega_S(M)$
can we recover $M$ if we know $M \in V$?
For generic $M$ the answer to this question only depends on the observed entries $S$
and not the particular values observed.
Namely, given $\Omega(M)$ for generic $M \in V$,
$M$ may be recovered up to finitely many choices
if and only if $S$ is a spanning set of the algebraic matroid associated to $V$.
Hence it is useful to find combinatorial descriptions of the
algebraic matroids associated to various determinantal varieties.
We obtain such combinatorial descriptions for the cases
where $V = \cals_2^n(\cc)$ and $V = \calm_2^{m\times n}(\cc)$.
The most natural way to phrase our characterization is in terms of the independent sets of $V$.
Note that a subset $S$ of entries is independent in the algebraic matroid underlying $V$
if and only if $\Omega_S: V \rightarrow \cc^S$ is a dominant morphism.
The salient feature for independent sets is that $\cc\setminus\Omega_S(V)$
has Lebesgue measure zero.
We also note that our result here answers a question of 
Kalai, Nevo, and Novik in \cite{kalai2016bipartite}
to find a combinatorial classification of what they call ``minimally $(2,2)$-rigid graphs''
(posed in the paragraph after the proof of their Example 5.5).
Using our language, these are the maximal independent sets in the algebraic matroid underlying $\calm_2^{m\times n}(\cc)$.
\\
\indent
An alternating closed trail in a directed graph is
a walk $v_0,v_1,\dots,v_k$ such that each
edge appears at most once,
$v_k = v_0$,
and adjacent edges $v_{i-1}v_i$ and $v_{i}v_{i+1}$ have opposite orientations (indices taken modulo $k+1$).
We now state our main result.

\newtheorem*{thm:introMatrices}{Theorems \ref{thm:skewMatroid} and \ref{thm:rectangularMatroid}}
\begin{thm:introMatrices}
    Let $V = \cals^n_2(\cc)$ be the variety of skew-symmetric $n\times n$ matrices of rank at most $2$,
    or $V = \calm_2^{m\times n}(\cc)$ be the variety of rectangular $m \times n$ matrices of rank at most $2$.
    A subset of observed entries $S \subseteq \binom{[n]}{2}$ (skew symmetric case)
    or $S \subseteq [m]\times [n]$ (rectangular case)
    is independent in the algebraic matroid underlying $V$
    if an only if there exists some acyclic orientation of $G(S)$ that has no alternating closed trail.
\end{thm:introMatrices}

Using techniques of \cite{kiraly-theran-tomioka2015},
one can see that deciding whether a given $S \subseteq \binom{[n]}{2}$ or $S \subseteq [m]\times [n]$
{ is independent in the algebraic matroid underlying $\cals_{2}^n(\cc)$ or $\calm_2^{m\times n}(\cc)$}
is in the complexity class RP.
Hence both decision problems are also in NP.
We provide an explicit combinatorial certificate of this fact.
Existence of a polynomial time algorithm for solving either decision problem remains open.
\\
\indent
The key to our approach is to use tropical geometry
to allow us to reduce to the following easier question:
Which entries of a dissimilarity map may be arbitrarily specified
such that the resulting partial dissimilarity map may be completed to a tree metric?
Our result here is as follows.

\newtheorem*{thm:introTrees}{Theorem \ref{thm:treeMatroid}}
\begin{thm:introTrees}
    Let $S \subseteq \binom{[n]}{2}$.
    Any partial dissimilarity map
    whose known distances are given by $S$
    can be completed to a tree metric, regardless of what those specified values are,
    if and only if there exists some acyclic orientation of $G(S)$ that has no alternating closed trail.
\end{thm:introTrees}

As in the case of partial matrices,
we give an explicit combinatorial certificate showing that the corresponding decision problem is in NP
but we do not know whether a polynomial time algorithm exists.
\\
\indent
The problem of deciding whether a \emph{particular} partial dissimilarity map
is completable to a tree metric was shown to be NP-complete
in \cite{farach1995robust}.
Note that this is distinct from our decision problem here,
because we are not setting values of the observed entries.
Special cases that allow a polynomial time algorithm were investigated in \cite{guenoche2004extension,guenoche2001triangles}.
Questions about whether a partial dissimilarity map can be completed to a tree metric
with a particular topology have been addressed in \cite{dress2014matroid,dress2012lassoing}.
\\
\indent
The outline of the paper is as follows.
Section \ref{sec:tropLemmas} lays out some general theory for using
tropical geometry to characterize algebraic matroids.
Section \ref{sec:trees} contains our results relating to completion of tree metrics.
Section \ref{sec:matrices} shows how our results on partial matrices
are easily obtained from our results on trees.

\section*{Acknowledgments}
    The author learned about matrix completion from Louis Theran at the Aalto Summer School in Algebra, Combinatorics and Statistics and is grateful to him for this and
    for many helpful conversations thereafter.
    He is also grateful to Bernd Sturmfels whose Algebraic Fitness 
    worksheet given at the Aalto Summer School
    asked a question about skew-symmetric matrix completion and its tropicalization.
    Finally, he is grateful to Seth Sullivant for his encouragement on this project,
    for many helpful conversations, and for feedback on early drafts.
    The author was partially supported by the US National Science Foundation (DMS 0954865) and the David and Lucille Packard Foundation.

\section{Completion and tropical varieties}\label{sec:tropLemmas}
We begin with the necessary preliminaries from tropical geometry.
The most important parts of this section are Lemmas \ref{lem:tropicalAlgebraicMatroid}
and \ref{lem:basisTransfer} which enable us to use the polyhedral
structure of a tropical variety to gain insight into 
the corresponding algebraic matroid.

Let $\kk$ be a field and let $V \subset \kk^n$.
We let $M(V)$ denote the \emph{independence complex} of $S$
which we define to be the collection of subsets $S$ of $\{1,\dots,n\}$
such that the projection of $V$ onto the coordinates indicated by $S$
is full-dimensional in $\kk^S$.
In the case that the sets in $M(V)$ form a matroid
(e.g. $V$ is an irreducible affine variety)
we refer to $M(V)$ as the \emph{matroid underlying $V$}.

Denote by $\puis$ the field of complex formal Puiseux series.
That is, $\puis$ is the set of all formal sums $\sum_{\alpha \in J} c_\alpha t^\alpha$
where $J \subset \qq$ such that $J$ has a smallest element
and the elements of $J$ can be expressed over a common denominator.
The valuation map $\val(\cdot): \puis\rightarrow \qq$
sends $\sum_{\alpha \in J} c_\alpha t^\alpha$ to $\min\{\alpha \in J: c_{\alpha} \neq 0\}$.
For any affine variety $V$ over $\puis$,
the corresponding \emph{tropical variety} is
\[
    \trop(V) = \overline{\{(-\val(x_1),\dots,-\val(x_n)): (x_1,\dots,x_n) \in V\}} \subseteq \rr^n
\]
where the overline indicates closure in the Euclidean topology on $\rr^n$.
We sometimes refer to $\trop(V)$ as \emph{the tropicalization of $V$}.
\\
\indent
We can also tropicalize varieties over $\cc$ by lifting to $\puis$ and 
tropicalizing there..
More specifically, let $V \subseteq \cc^n$ be an affine variety over $\cc$
with ideal $I \subseteq \cc[x_1,\dots,x_n]$.
By lifting this ideal into $\puis[x_1,\dots,x_n]$ we obtain a variety $V' \subseteq (\puis)^n$.
The tropical variety $\trop(V)$ corresponding to $V$ is simply $\trop(V')$.

\begin{lemma}[\cite{yu2015algebraic}, Lemma 2]\label{lem:yuTropAlgebraic}
    Let $V$ be an irreducible affine variety over either $\cc$ or $\puis$.
    Then the independence complex of $V$ and $\trop(V)$ are the same.
\end{lemma}

Many matrix completion problems ask for a combinatorial description
of the algebraic matroid associated to a particular irreducible affine variety.
Lemma \ref{lem:yuTropAlgebraic} says that we can tackle such problems
by looking at the corresponding tropical variety instead.
The advantage in this is that tropical varieties have a useful polyhedral structure
which we now describe.

\begin{defn}
    Let $\Sigma$ be a rational fan in $\rr^n$ of pure dimension $d$.
    We say that $\Sigma$ is \emph{balanced} if we can associate a positive integer $m(\sigma)$,
    called the \emph{multiplicity}
    to each full-dimensional cone $\sigma$
    in a way such that for each cone $\tau \in \Sigma$ of dimension $d-1$,
    \[
        \sum_{\sigma \supsetneq \tau} m(\sigma)v_\sigma \in \span(\tau)
    \]
    where $v_\sigma$ is the first lattice point on the ray $\sigma / \span(\tau)$.
    We say that $\Sigma$ is \emph{connected through codimension one}
    if for any $d$-dimensional cones $\sigma, \rho \in \Sigma$,
    there exists a sequence of $d$-dimensional cones $\sigma = \sigma_0,\sigma_1,\dots,\sigma_k = \rho \in \Sigma$
    such that $\sigma_i \cap \sigma_{i+1}$ has dimension $d-1$.
\end{defn}

The well-known structure theorem for tropical varieties
applied to the special case where the defining equations
have constant coefficients gives us the following.

\begin{thm}[\cite{maclagan2015introduction}, Theorem 3.3.5]\label{thm:tropicalStructure}
    Let $V$ be an irreducible $d$-dimensional affine variety over $\cc$.
    Then $\trop(V)$ is the support of a balanced fan of pure dimension $d$
    that is connected through codimension one.
\end{thm}

The following two lemmas give us ways that one can use the polyhedral structure
of $\trop(V)$ to gain insight into the structure of $M(V)$.
When we refer to a cone in $\trop(V)$,
we mean in a polyhedral subdivision of $\trop(V)$ that is balanced
and connected through codimension one.

\begin{lemma}\label{lem:tropicalAlgebraicMatroid}
    Let $V \subseteq \cc^n$ be an irreducible affine variety of dimension $d$.
    Then $S \subseteq [n]$ is independent in $M(V)$ if and only if
    $S$ is independent in $M(\span(\sigma))$ for some $d$-dimensional
    cone $\sigma$ in $\trop(V)$.
\end{lemma}
\begin{proof}
    Let $S$ be independent in $M(V)$.
    By Lemma \ref{lem:yuTropAlgebraic}, the projection of
    $\trop(V)$ onto $\rr^S$ is full dimensional.
    In particular, the projection of some maximal dimensional cone
    $\sigma \in \trop(V)$ onto $\rr^S$ is full dimensional and therefore $S$
    is independent in the matroid $M(\span(\sigma))$.
    \\
    \indent
    Now let $S$ be independent in $M(\span(\sigma))$ for some maximal cone $\sigma \in \trop(V)$.
    Then the projection of $\span(\sigma)$ onto $\rr^S$ is full dimensional.
    Therefore the same holds for $\sigma$ and therefore $\trop(V)$.
    So $S$ is in the independence complex of $\trop(V)$
    and therefore independent in $M(V)$ by Lemma \ref{lem:yuTropAlgebraic}.
\end{proof}

\begin{lemma}\label{lem:basisTransfer}
    Let $V \subseteq \cc^n$ be an irreducible $d$-dimensional affine variety.
    Let $\tau$ be a $d-1$ dimensional cone of $\trop(V)$
    and let $\sigma_1,\dots,\sigma_k$ be the $d$-dimensional cones in $\trop(V)$
    containing $\tau$.
    If $B \subseteq [n]$ is a basis of $M(\span(\sigma_1))$ then
    $B$ is also a basis of $M(\span(\sigma_i))$ for some $i \neq 1$.
\end{lemma}
\begin{proof}
    Let $L = \span(\tau)$ and $L_i = \span(\sigma_i)$.
    By Theorem \ref{thm:tropicalStructure} there exist $v_i \in \sigma_i\setminus L$
    such that $\sum_{i=1}^k v_i \in L$.
    Let $B$ be a basis of $L_1$.
    Then the projection of $L_1$ onto $\rr^B$ is all of $\rr^B$
    and the projection of $L$ onto $\rr^B$ is some hyperplane
    $\{x \in \rr^B : c^Tx = 0\}$.
    By padding with extra zeros, we can extend $c$ to an element of $\rr^n$.
    Then $L = L_1 \cap \{x \in \rr^n : c^Tx = 0\}$ and $c^Tv_i \neq 0$.
    Since $\sum_{i=1}^k v_i \in L$,
    we must have $\sum_{i=1}^k c^T v_i = 0$.
    Therefore there must exist some $i \neq 1$ such that $c^Tv_i \neq 0$.
    Since $L_i = L + \span(v_i)$, the projection of $L_i$
    onto $\rr^B$ is all of $\rr^B$.
    So $B$ is a basis of $L_i$.
\end{proof}

We end this section by noting a nice feature about projections
of tropical varieties which was noted by Yu in \cite{yu2015algebraic}.

\begin{prop}\label{prop:projFullDimensional}
    Let $\kk$ be either $\cc$ or $\puis$ and let $V \subseteq \kk^n$ be an 
    irreducible affine variety.
    If $S$ is independent in $M(V)$ then the projection of $\trop(V)$
    onto $\rr^S$ is all of $\rr^S$.
\end{prop}

\section{Tree metrics and tree matroids}\label{sec:trees}
In this section we determine
which entries of a dissimilarity map can be arbitrarily specified
while still allowing completion to a tree metric.
We begin with the necessary preliminaries on tree metrics.
Proposition \ref{prop:treesTropStructure} describes the tropical
and polyhedral structure on the set of tree metrics.
Lemma \ref{lemma:caterpillar} reduces our question about completion
to an arbitrary tree metric to the question about completion to a tree
metric whose topology is a caterpillar tree.
We answer this simpler question in Proposition \ref{prop:caterpilarMatroid} and give 
the section's main result in Theorem \ref{thm:treeMatroid}.
\\
\indent
We now give the necessary background on tree metrics.
For a more detailed account, see \cite{semple2003phylogenetics}.
Let $X$ be a set.
A \emph{dissimilarity map on $X$} is a function $\delta : X\times X \rightarrow \rr$
such that $\delta(x,x) = 0$ and $\delta(x,y) = \delta(y,x)$ for all $x,y \in X$.
If $X = \{x_1,\dots,x_n\}$ and $[n] = \{1,\dots,n\}$
then there is a natural bijection between dissimilarity maps on $X$
and points in $\rr^{\binom{[n]}{2}}$.
Namely, denoting the entry of a point $d \in \rr^{\binom{[n]}{2}}$ corresponding to $\{i,j\}$
by $d_{ij}$ or $d_{ji}$,
we associate a dissimilarity map $\delta$ on $X$ with the point $d \in \rr^{\binom{[n]}{2}}$
such that $d_{ij} = \delta(x_i,x_j)$.
Hence we will often speak of dissimilarity maps
as if they were merely points in $\rr^{\binom{[n]}{2}}$.
\\
\indent
A tree with no vertices of degree $2$ with leaf set $X$ is called
an $X$-tree.
Let $T$ be an $X$-tree and let $w$ be an edge-weighting on $T$
such that $w(e) > 0$ for all internal edges $e$ of $T$.
Some authors require that $w(e)> 0$ for \emph{any} edge $e$ of $T$,
but we do not.
The triple $(T,X,w)$ gives rise to a dissimilarity map
$d$ where $d_{ij}$ is the sum of the edge weights given by $w$
along the unique path from $x_i$ to $x_j$ in $T$.
Any dissimilarity map $d$ that arises from a triple in this way is called a \emph{tree metric}.
For example, if $X = \{1,2,3,4\}$,
then the dissimilarity map
$d = (d_{12},d_{13},d_{14},d_{23},d_{24},d_{34}) = (0,3,-2,5,0,-1)$
is a tree metric because it can be displayed on a tree as in Figure \ref{fig:treeMetric}.
Given a tree metric $d$ on $X$, there is a unique $X$-tree $T$
that realizes $d$ as such \cite[Theorem 7.1.8]{semple2003phylogenetics}.
This tree $T$ is called the \emph{topology} of $d$.
A tree $T$ is \emph{binary} if all internal vertices have degree three.

\begin{figure}
    \includegraphics[scale=0.4]{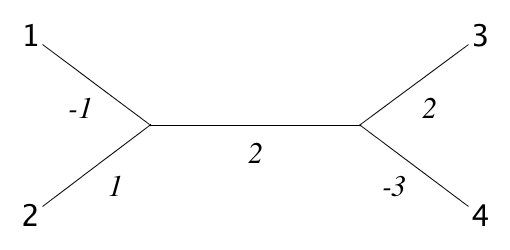}
    \caption{An edge-weighted tree with all internal edge weights positive
    whose leaves are labeled by $\{1,2,3,4\}$ showing that $d = (d_{12},d_{13},d_{14},d_{23},d_{24},d_{34}) = (0,3,-2,5,0,-1)$ is a tree metric}\label{fig:treeMetric}
\end{figure}

\begin{prop}[\cite{semple2003phylogenetics}, Theorem 7.2.6]
    A dissimilarity map $d \in \rr^{\binom{[n]}{2}}$ is a tree metric if and only if
    it satisfies
    \[
        d_{ij} + d_{kl} \le \max\{d_{ik} + d_{jl},d_{il} + d_{jk}\}
    \]
    for all distinct $i,j,k,l \in [n]$.
\end{prop}

Authors who require that all edge weights be positive in the definition
of a tree metric have a similar four point condition.
The only difference is that $i,j,k,l$ need not be distinct.
We denote the set of all tree metrics on a set of size $n$ by $\calt_n$.
\\
\indent
The following proposition summarizes many known facts
about the polyhedral and tropical structure of $\calt_n$ (references are given in the proof).
In particular, it tells us how to realize $\calt_n$ as a tropical variety
along with the polyhedral subdivision guaranteed to exist by Theorem \ref{thm:tropicalStructure}.
{ We note that it would not be true had we taken the more restrictive definition of tree
metric that required all (as opposed to just internal) edge weights to be positive.}
Recall that the Grassmannian $\gr_{k,n} \subset \kk^{\binom{[n]}{k}}$ is the irreducible affine variety
parameterized by the $k\times k$ minors of a $k\times n$ matrix over $\kk$.
\begin{prop}\label{prop:treesTropStructure}
    The space of phylogenetic trees $\calt_n$ is $\trop(\gr_{2,n})$.
    We can give $\calt_n$ a balanced fan structure 
    connected through codimension one as follows.
    Each open cone is the collection of tree metrics whose topology is $T$
    for a particular tree $T$.
    Such a cone is maximal if and only if $T$ is binary.
    The cone has codimension one if and only if
    $T$ can be obtained from a binary tree by contracting exactly one edge.
    Given a cone $\tau$ of codimension one corresponding to tree $T$,
    the cones containing $\tau$ correspond to the three binary trees
    that can be contracted to $T$.
\end{prop}
\begin{proof}
    Let $L_n$ be the set of tree metrics whose topology is a star tree -
    that is, a tree with no internal edges.
    Since an edge-weighting of a leaf-labeled tree gives rise to a tree metric
    if and only if all \emph{internal} edge weights are nonnegative,
    $L_n$ is the lineality space of $\calt_n$.
    Theorem 3.4 in \cite{speyer2004tropical} implies that $\calt_n / L_n = \trop(\gr_{2,n}) / L_n$.
    But as $L_n$ is also the lineality space of $\gr_{2,n}$ (c.f. remarks following
    Corollary 3.1 in \cite{speyer2004tropical}),
    this implies that $\calt_n = \trop(\gr_{2,n})$.
    \\
    \indent
    The polyhedral structure is given in \cite{billera2001geometry}.
    Connectedness through codimension one follows from the fact
    that any two binary trees on the same leaf set can be reached from one another
    via a finite sequence of nearest-neighbor-interchanges (see \cite[Theorem 2]{waterman1978similarity}).
    We can see that this polyhedral fan is balanced by assigning multiplicity $1$
    to each maximal cone (see Remark 4.3.11 and Theorem 3.4.14 in \cite{maclagan2015introduction}).
\end{proof}
Let $T$ be a tree without vertices of degree two whose leaves are labeled by $[n]$.
Denote its edge set by $E$ and let $\rr^E$ denote the vector space of edge-weightings of $T$.
For each pair $ij \in \binom{[n]}{2}$,
define $\lambda_T(ij) \in \rr^E$ by $\lambda_T(ij)_e = 1$ whenever $e$ is on the unique path from
the leaf labeled $i$ to the leaf labeled $j$
and $0$ otherwise.
We denote the linear matroid underlying these $\lambda_T(ij)$ by $M(T)$
and call it the \emph{matroid associated to $T$}.
We will abuse language and say that a set $S \subseteq \binom{[n]}{2}$
is independent in $M(T)$ to mean that $\{\lambda_T(ij)\}_{ij \in S}$ is independent in $M(T)$.
Let $A_T$ be the matrix whose columns are the $\lambda_T(ij)$s.
An example is given in Figure \ref{fig:treeMatroidExample}.
Note that $\rowspan(A_T)$ is the linear span of the cone in $\calt_n$
containing all tree metrics with topology $T$.
Moreover, $M(T)$ is the matroid of this linear space.
These matroids were introduced and studied in \cite{dress2014matroid}.

\begin{figure}
    \begin{subfigure}{0.49\textwidth}\centering
        \includegraphics[scale=0.4]{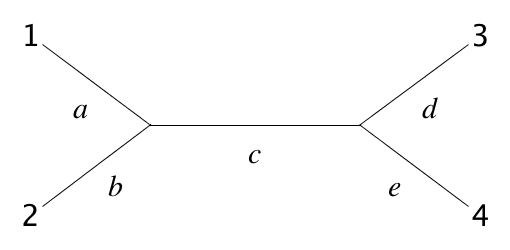}
    \end{subfigure}
    \begin{subfigure}{0.49\textwidth}\centering
        \[
            \bordermatrix{
                ~&12&13&14&23&24&34\cr
                a&1&1&1&0&0&0\cr
                b&1&0&0&1&1&0\cr
                c&0&1&1&1&1&0\cr
                d&0&1&0&1&0&1\cr
                e&0&0&1&0&1&1\cr
            }
        \]
    \end{subfigure}
    \caption{Let $T$ be the tree on the left with leaves \{1,2,3,4\} and edges labels $\{a,b,c,d,e\}$.
    The matrix on the right is $A_T$. Its columns are the $\lambda_T(ij)s$}
    \label{fig:treeMatroidExample}
\end{figure}

Let $T$ be a tree with leaf labels $[n]$.
A \emph{cherry} on a binary $[n]$-tree $T$ is a pair of leaves $ij$ such that
$i$ and $j$ are adjacent to the same vertex of degree $3$.
A caterpillar tree on $n \ge 4$ vertices is a binary tree with exactly two cherries.
See Figure \ref{fig:cherries} for an illustration.

\begin{figure}
    \begin{subfigure}{0.49\textwidth}
        \includegraphics[scale=0.4]{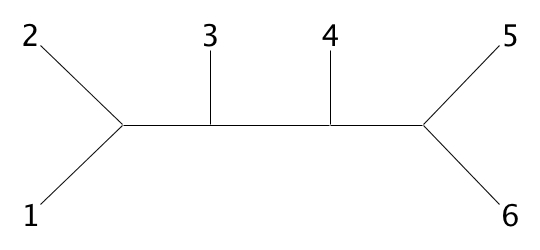}\hfill
    \end{subfigure}
    \begin{subfigure}{0.49\textwidth}
        \hfill\includegraphics[scale=0.4]{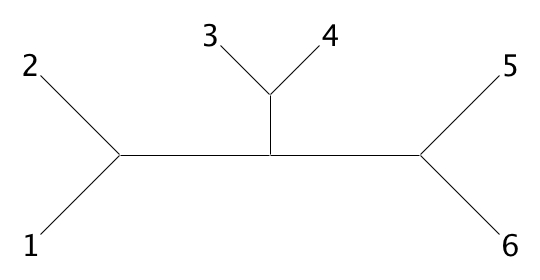}
    \end{subfigure}
    \caption{The tree on the left has cherries $12$ and $56$ hence it is a caterpillar.
    The tree on the right has cherries $12$, $34$ and $56$ and is therefore \emph{not}
    a caterpillar.}
    \label{fig:cherries}
\end{figure}

\begin{lemma}\label{lemma:caterpillar}
    A set $B \subseteq \binom{[n]}{2}$ is a basis of $M(\calt_n)$
    if and only if $B$ is a basis in the matroid
    associated to some caterpillar tree.
\end{lemma}
\begin{proof}
    Lemma \ref{lem:tropicalAlgebraicMatroid} implies that $S$ is independent in $M(\calt_n)$
    if and only if $S$ is independent in $M(T)$ for some binary tree $T$.
    We proceed by showing that if $B$ is a basis of $M(T)$ for a non-caterpillar binary tree $T$,
    then $B$ is also the basis of $M(T')$ for some binary tree $T'$ with one fewer cherry.
    It will follow by induction that $B$ is a basis of $M(C)$ for some caterpillar $C$.
    \\
    \indent
    So let $T$ be a binary tree with three distinct cherries $ij$, $i'j'$, $i''j''$
    with corresponding degree-$3$ vertices $p,p',p''$.
    There is a single vertex $q$ in the common intersection of the three paths
    $p$ to $p'$, $p'$ to $p''$, and $p$ to $p''$.
    If we delete $q$ from $T$ we get three trees $U,V,W$ which contain
    the cherries $ij$, $i'j'$, and $i''j''$ respectively.
    \\
    \indent
    We now induct on the number of leaves in $U$.
    The base case is where $U$ consists of the two leaves $i$ and $j$.
    We can visualize this as in Figure \ref{fig:T}
    letting $U'$ be $i$ and $U''$ be $j$.
    Lemma \ref{lem:basisTransfer} and Proposition \ref{prop:treesTropStructure}
    imply that $B$ must also be a basis of one of the trees $T_1$ or $T_2$,
    depicted in Figures \ref{fig:T1} and \ref{fig:T2} respectively.
    Note that $T_1$ and $T_2$ each have fewer cherries than $T$.
    \\
    \indent
    Now assume $U$ has more than two leaves.
    Let $s$ be the first node on the path from $q$ to $p$.
    Deleting $s$ from $U$ splits it into two subtrees
    which we denote $U'$ and $U''$.
    This is depicted in Figure \ref{fig:T}.
    Assume $ij$ belongs to the subtree $U'$.
    Using Lemma \ref{lem:basisTransfer} and Proposition \ref{prop:treesTropStructure}
    as before,{}
    we see that $B$ must also be a basis of one of the trees $T_3$ or $T_4$,
    depicted in Figures \ref{fig:T1} and \ref{fig:T2} respectively.

    \begin{figure}[h]
    \begin{subfigure}{0.33\textwidth}\centering
        \includegraphics[scale=0.25]{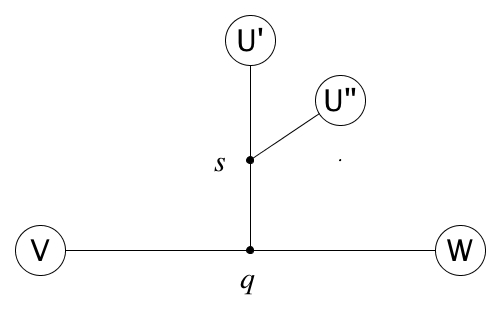}
        \caption{$T$}\label{fig:T}
    \end{subfigure}\hspace*{\fill}
    \begin{subfigure}{0.33\textwidth}\centering
        \includegraphics[scale=0.25]{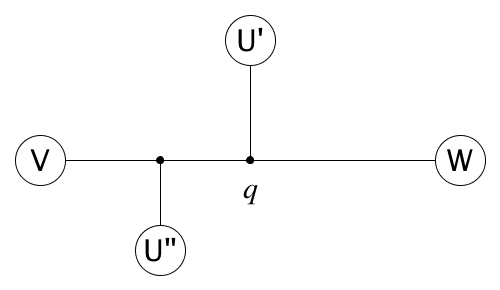}
        \caption{$T_3$}\label{fig:T1}
    \end{subfigure}\hspace*{\fill}
    \begin{subfigure}{0.33\textwidth}\centering
        \includegraphics[scale=0.25]{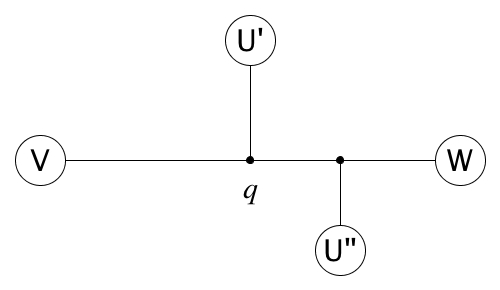}
        \caption{$T_4$}\label{fig:T2}
    \end{subfigure}
    \caption{Breaking $T$ into subtrees}\label{fig:breakIntoSubtrees}
    \end{figure}

    We can now repeat the process with $U'$ taking the place of $U$.
    In the case that $B$ is a basis of $T_3$ we expand $V$ to include $U''$,
    and in the case of $T_4$ we instead expand $W$ to contain $U''$.
    The result has strictly fewer leaves in $U$ and so we induct.
\end{proof}

Lemma \ref{lemma:caterpillar} reduces the problem of describing the independent sets of $M(\calt_n)$
to the simpler problem of describing the independent sets of $M(T)$ where $T$ is a caterpillar tree.
Luckily, such independent sets have a simple combinatorial description.
Given a subset $S \subseteq \binom{[n]}{2}$, we let $G(S)$ denote the graph with vertex set
$[n]$ and edge set $S$.
For each $n$, we let $\cat(n)$ denote the caterpillar 
with cherries $1,2$ and $n-1,n$ such that the other leaf labels
increase from $3$ to $n-2$ along the path from the $1,2$ cherry to the $n-1,n$ cherry.
This is depicted in Figure \ref{fig:caterpillar}.
Recall that a \emph{closed trail} in a graph is a sequence of vertices $v_0,\dots,v_k$
such that each $(v_i,v_{i+1})$ is an edge (indices taken mod $k+1$),
no edge is repeated, and $v_0 = v_k$.

\begin{figure}[h]
    \includegraphics[scale=0.4]{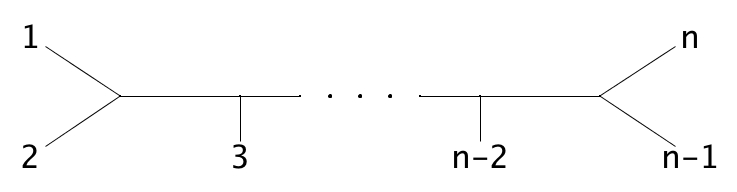}
    \caption{The caterpillar $\cat(n)$}\label{fig:caterpillar}
\end{figure}

\begin{prop}\label{prop:caterpilarMatroid}
    Let $S \subseteq \binom{[n]}{2}$.
    Then $S$ is independent in $M(\cat(n))$ if and only if $G(S)$ contains
    no closed trail with alternating vertices.
    That is, no closed trail of the form $v_0,v_1,\dots,v_{2k-1}$ where
    $v_{2i-1},v_{2i+1} < v_{2i}$ for each $i=0,\dots,k-1$ where indices are taken mod $2k$. 
\end{prop}
\begin{proof}
    Let $A_{\cat(n)}$ be the matrix whose columns are given by the $\lambda_{\cat(n)}(ij)$s.
    We claim that if $x \in \rr^{\binom{[n]}{2}}$ then $x \in \rowspan(A_{\cat(n)})$ if and only if
    $x_{ik} + x_{jl} - x_{il}-x_{jk} = 0$ for all $i<j<k<l$.
    To see this, note that $\rowspan(A_{\cat(n)})$ is equal to the linear hull of
    the collection of tree metrics with topology $\cat(n)$.
    If $x$ is a tree metric with topology $\cat(n)$, then each quartet $i<j<k<l$ must satisfy
    $x_{ij} + x_{kl} < x_{ik} + x_{jl} = x_{il} + x_{jk}$.
    Since the topology of a tree metric is determined by its quartets \cite{semple2003phylogenetics}
    the relations $x_{ik} + x_{jl} - x_{il} - x_{jk} = 0$ define the linear hull of
    the set of tree metrics with topology $\cat(n)$.
    \\
    \indent
    Now let $B$ be the vertex-edge incidence matrix of a complete bipartite graph on $n$ vertices. 
    That is, the columns of $B$ are indexed by the set $[n]\times [n]$,
    the rows are indexed by two disjoint copies of $[n]$,
    and the column corresponding to $(i,j)$
    has $1$s at the entries corresponding to $i$ in the first copy of $[n]$ and $j$ in the second copy,
    and $0$s elsewhere.
    Then $x \in \rowspan(B)$ if and only if $x_{(i,k)} + x_{(j,l)} - x_{(i,l)}-x_{(j,k)} = 0$
    for all $i,j,k,l$.
    \\
    \indent
    Notice that any functional vanishing on $\rowspan(A_{\cat(n)})$ can be associated to a functional vanishing on $\rowspan(B)$
    in a coordinate-wise way.
    In particular, for $i<j$ we associate each coordinate $ij$ in $\rowspan(A_{\cat(n)})$ with the coordinate $(i,j)$ in $\rowspan(B)$.
    This means that $M(\cat(n))$, which is the column matroid of $A_{\cat(n)}$, is the restriction of the column matroid of $B$
    to ground set $\{x_{(i,j)}\}$ where $i < j$.
    \\
    \indent
    The column matroid of $B$ is the polygon matroid on the complete bipartite graph
    on two disjoint copies of $[n]$.
    To realize $M(\cat(n))$ as a submatroid,
    we restrict to the collection of edges $(i,j)$ from $\{1,\dots,n-1\}$ to $\{2,\dots,n\}$ such that $i < j$.
    Closed trails supported on such edges in this bipartite graph correspond exactly to alternating closed trails
    in the complete graph on vertex set $[n]$.
\end{proof}

\begin{rmk}\label{rmk:treeCertificate}
    We now have a purely combinatorial proof that
    the problem of deciding whether a given $S \subseteq \binom{[n]}{2}$
    is independent in $M(\calt_n)$
    is in the complexity class NP.
    Namely, let $H_n$ be the bipartite graph on partite sets $A=\{1,\dots,n-1\}$
    and $B=\{2,\dots,n\}$ where $(i,j)$ is an edge from $A$ to $B$ if and only if $i < j$.
    By Proposition \ref{prop:caterpilarMatroid},
    $S \subseteq \binom{[n]}{2}$ is independent in $M(\cat_n)$
    if and only if the corresponding edges in $H_n$ form no cycles.
    Therefore, by Lemma \ref{lemma:caterpillar},
    a polynomial-time verifiable certificate
    that a given $S \subseteq \binom{[n]}{2}$ is independent in $M(\calt_n)$
    is a permutation $\sigma$ of $[n]$
    such that the edges $\{(\sigma(i),\sigma(j), (i,j) \in S\}$ of $H_n$
    induce no cycles.
\end{rmk}

Recall that an acyclic orientation of a graph is an assignment of
directions to each edge in a way such that produces no directed cycles.
An alternating closed trail in a directed graph is a closed trail $v_0,\dots,v_k$
such that each pair of adjacent edges $v_{i-1}v_i$ and $v_{i}v_{i+1}$
have opposite orientations.
We remind the reader of Proposition \ref{prop:projFullDimensional}
which implies that if $S \subseteq \binom{[n]}{2}$ is independent in $M(\calt_n)$,
then one may arbitrarily specify the distances among pairs in $S$
and still be able to extend the result to a tree metric.
We are now ready to prove Theorem \ref{thm:treeMatroid}
which we restate below using the language of this section.

\begin{thm}\label{thm:treeMatroid}
    Let $S \subseteq \binom{[n]}{2}$.
    Then $S$ is independent in $M(\calt_n)$ if and only if
    some acyclic orientation of $G(S)$ has no alternating closed trails.
\end{thm}
\begin{proof}
    Assume that some acyclic orientation of $G(S)$ has no alternating closed trails.
    Fix such an acyclic orientation.
    Choose a permutation $\sigma: [n]\rightarrow [n]$ such that 
    edge $ij$ is oriented from $i$ to $j$ if and only if $\sigma(i) < \sigma(j)$.
    Then by Proposition \ref{prop:caterpilarMatroid},
    applying $\sigma$ to the vertices of $G(S)$ gives an independent set in $M(\cat(n))$.
    Since independence in $M(\calt_n)$ is invariant under permutation of the leaves,
    this implies that $S$ is independent in $M(\calt_n)$.
    \\
    \indent
    Now assume that every acyclic orientation of $G(S)$ has an alternating closed trail.
    Then there does not exist any permutation of $[n]$ that
    produces an independent set of $M(\cat(n))$ when applied to $S$.
    Therefore $S$ is not independent in $M(T)$ for any caterpillar $T$.
    Lemma \ref{lemma:caterpillar} then implies that $S$ is not independent in $M(\calt_n)$.
\end{proof}

\section{Rank two matrices}\label{sec:matrices}
We show how Theorem \ref{thm:treeMatroid} immediately
characterizes the algebraic matroid underlying the set of
$n\times n$ skew-symmetric matrices of rank at most $2$.
We then use this to characterize the algebraic matroid underlying the set
of general $m\times n$ matrices of rank at most $2$.

Denote by $\cals_r^n(\kk)$ the collection of $n\times n$
skew-symmetric $\kk$-matrices of rank at most $r$
and denote by $\calm_{r}^{m\times n}(\kk)$ the collection of $m\times n$ $\kk$-matrices of rank at most $r$.
Both are irreducible algebraic varieties.
We will only be concerned with these sets when $\kk$ is $\cc$ or $\puis$.
Lemma \ref{lem:yuTropAlgebraic} implies that limited to these choices,
the algebraic matroid does not depend on the ground field
and so we will suppress $\kk$ from notation.

\begin{prop}\label{prop:skewtrop}
    The tropicalization of the set of $n\times n$ skew symmetric $\puis$-matrices
    of rank at most $2$ is the set of all tree metrics on $n$ leaves.
    That is, $\trop(\cals_2^n(\cc)) = \calt_n$.
\end{prop}
\begin{proof}
	It is shown in \cite{speyer2004tropical} that $\trop(\gr_{2,n}) = \calt_n$.
    Therefore the proposition follows from the claim that
    the projection of $\cals_2^n(\cc)$ onto the
    upper-triangular coordinates is $\gr_{2,n}$.
    This claim seems to be a well-known fact,
    but it is difficult to find a precise reference so we give a proof here.
    \\
    \indent
    {
    Let $\{x_{ij}: 1 \le i < j \le n\}$ be indeterminates and let $M$ be the $n\times n$ skew-symmetric matrix
    whose $ij$ entry is $x_{ij}$ whenever $i < j$.
    Let $I_{2,n} \subseteq \puis [x_{ij}: 1 \le i<j \le n]$ be the radical of the ideal generated by the $3\times 3$ minors of $M$.
    The variety of $I_{2,n}$ is the projection
    onto the upper-triangular coordinates of the set of all $n\times n$ skew-symmetric matrices of rank at most $2$.
    Moreover, $I_{2,n}$ contains all principal $4\times 4$ minors of $M$,
    and these polynomials are the squares of the canonical generating set of $\gr_{2,n}$.
    Therefore $\gr_{2,n}$ contains the variety of $I_{2,n}$.
    Since $\gr_{2,n}$ is irreducible, as it is parameterized by the $2\times 2$ determinants of a generic $2\times n$ matrix,
    it can be seen to be equal to the variety of $I_{2,n}$ by showing that the variety of $n\times n$ skew-symmetric
    matrices of rank at most $2$ has dimension at least that of $\gr_{2,n}$.
    Since the dimension of $\calt_n$ is $2n-3$ (see e.g. \cite{semple2003phylogenetics}),
    Theorem \ref{thm:tropicalStructure} implies that $\gr_{2,n}$ has dimension $2n-3$ as well.
    We can see that the set of skew-symmetric matrices of rank at most $2$ has at least this dimension 
    because in constructing such a matrix we may arbitrarily specify all but the first entry of the first row
    and all but the first and second entries of the second row.
    }
\end{proof}

\begin{thm}
\label{thm:skewMatroid}
    Let $S \subseteq \binom{[n]}{2}$.
    Then $S$ is independent in $M(\cals_2^n)$ if and only if
    some acyclic orientation of $G(S)$ has no alternating closed trails.
\end{thm}
\begin{proof}
    This follows from Lemma \ref{lem:yuTropAlgebraic} and Proposition \ref{prop:skewtrop}
    and Theorem \ref{thm:treeMatroid}.
\end{proof}

\begin{rmk}\label{rmk:skewPolynomial}
    The same combinatorial polynomial-time verifiable certificate as in Remark \ref{rmk:treeCertificate}
    works to confirm that a given $S \subseteq \binom{[n]}{2}$ is
    independent in $M(\cals_2^n)$.
\end{rmk}

The dimension of $\cals_2^n$ is $2n-3$.
Therefore the rank of the matroid $M(\cals_2^n)$ is also $2n-3$.
Also note that $M(\cals_2^k)$ is a restriction of $M(\cals_2^n)$ whenever $k < n$.
It follows that if $S$ is a basis of $M(\cals_2^n)$,
then $G(S)$ has $2n-3$ edges and each subgraph of size $k$ has at most $2k-3$ edges.
It is a famous theorem of Laman \cite{laman1970graphs} that graphs satisfying these constraints
are exactly the minimal graphs which are generically infinitesimally 
rigid in the plane.
These graphs are often called ``Laman graphs.''
{ One might wonder whether all Laman graphs
are bases of $M(\cals_2^n)$
but this is not the case.}
A counterexample is $K_{3,3}$, the complete bipartite graph
on two partite sets of size three (see Figure \ref{fig:K33}).
This is a Laman graph but it is dependent in $M(\cals_2^n(\cc))$
because the coordinates specified by $K_{3,3}$
are the entries in a $3\times 3$ submatrix.
Such a submatrix must be singular in a matrix of rank $2$
and so those entries satisfy a polynomial relation.
Alternatively, one could appeal to Theorem \ref{thm:skewMatroid}
and check that every acyclic orientation of $K_{3,3}$ induces
an alternating cycle.

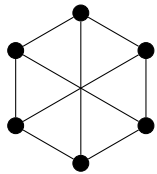
\begin{figure}[h]
    \begin{tikzpicture}
        \vertex (a) at (270:1)[]{};
        \vertex (b) at (330:1)[]{};
        \vertex (c) at (30:1)[]{};
        \vertex (d) at (90:1)[]{};
        \vertex (e) at (150:1)[]{};
        \vertex (f) at (210:1)[]{};
        \path
            (a) edge (b)
            (b) edge (c)
            (c) edge (d)
            (d) edge (e)
            (e) edge (f)
            (f) edge (a)
            (a) edge (d)
            (b) edge (e)
            (c) edge (f)
        ;
    \end{tikzpicture}
    \caption{$K_{3,3}$, a Laman graph that is not a basis of $M(\cals_2^n)$.}
    \label{fig:K33}
\end{figure}

{ As we will see in the proof of the following theorem,
$\calm_2^{m\times n}$ can be realized as a coordinate projection
of $\cals_2^k$ for any $k \ge m+n$}.
This implies that the matroid $M(\calm_2^{m\times n})$
can be realized as the restriction of such a $M(\cals_2^k)$.
We can use this fact to give a characterization of the algebraic matroid of $M(\calm_2^{m\times n})$,
whose bases are also the minimally $(2,2)$-rigid graphs in \cite{kalai2016bipartite}.
The ground set of $M(\calm_2^{m\times n})$ can be associated with the edges in the complete bipartite graph
with partite vertex sets $[m]$ and $[n]$, denoted $K_{m,n}$.

\begin{thm}
\label{thm:rectangularMatroid}
    Let $S$ be a collection of edges of $K_{m,n}$.
    Then $S$ is independent in $M(\calm_2^{m\times n})$ if and only if
    some acyclic orientation of $G(S)$ has no alternating closed trails.
\end{thm}
\begin{proof}
	{
	Define $k := m+n$.
	We claim that $\calm_2^{m\times n}$ is a coordinate projection of $\cals_2^k$.
	To see this, first let $A \in \calm_2^{m \times n}$.
	Then choose $u_1,u_2 \in \cc^{m\times 1}$ and $v_1,v_2 \in \cc^{1 \times n}$
	such that $A = u_1v_1 - u_2v_2$.
	Then define the following $k \times k$ matrix
	\[
		B :=
		\begin{pmatrix}
		u_1 \\
		v_2^T
		\end{pmatrix}
		\begin{pmatrix}
		u_2^T & v_1
		\end{pmatrix}
		-
		\begin{pmatrix}
		u_2 \\
		v_1^T
		\end{pmatrix}
		\begin{pmatrix}
		u_1^T & v_2
		\end{pmatrix}.
	\]
	Note that $B$ is skew-skew symmetric of rank at most $2$,
	and its upper-right block is equal to $A$.
    Thus the claim is proven
    and therefore $M(\calm_2^{m\times n})$ is a restriction of $M(\cals_2^k)$
    }
    The result then follows from Theorem \ref{thm:skewMatroid}.
\end{proof}

\begin{rmk}\label{rmk:rectangularCertificate}
    The content of Remarks
    \ref{rmk:treeCertificate} and \ref{rmk:skewPolynomial}
    can be adapted to give a combinatorial polynomial-time verifiable certificate
    that a given $S \subseteq [m]\times [n]$ is independent in $M(\calm_2^{m\times n})$.
    Namely, assuming $m > n$, we can translate $S$ into a subset $T$ of $\binom{[m+n]}{2}$
    by mapping each $(i,j) \in S$ to the pair $\{i+m,j\}$.
    Then the construction given in Remark \ref{rmk:treeCertificate}
    can be applied to $T$.
\end{rmk}

\bibliography{matrixCompletion}
\bibliographystyle{plain}{}

\end{document}